\documentclass[onecolumn,showpacs]{revtex4-2}
\usepackage{graphicx}
\usepackage{dcolumn}
\usepackage{amsmath}
\usepackage{amssymb}
\usepackage{hyperref}
\usepackage{scalerel}
\usepackage{cancel}
\allowdisplaybreaks

\newtheorem{theorem}{Theorem}[section]
\newtheorem{corollary}[theorem]{Corollary}
\newtheorem{proposition}[theorem]{Proposition}

\newtheorem{lemma}[theorem]{Lemma}
\newtheorem{definition}[theorem]{Definition}
\newenvironment{proof}[1][Proof]{\begin{trivlist}
\item[\hskip \labelsep {\bfseries #1}]}{\end{trivlist}}

\newcommand{\qed}{\nobreak \ifvmode \relax \else
	\ifdim\lastskip<1.5em \hskip- \lastskip
	\hskip1.5em plus0em minus0.5em \fi \nobreak
	\vrule height0.75em width0.5em depth0.25em\fi}

\begin{document}

\title{Conformal geometry on a class of embedded hypersurfaces in spacetimes}

\author{Abbas M. \surname{Sherif}$^{1,2,}$}
\email{abbasmsherif25@ibs.re.kr}
\affiliation{$^1$Center for Geometry and Phyics, Institute for Basic Sciences, Pohang University of Science and Technology, 77 Cheongham-ro, Nam-gu, Pohang, Gyeongbuk 37673, South Korea}

\author{Peter K. S. \surname{Dunsby}$^{2,3}$}
\email{peter.dunsby@uct.ac.za}
\affiliation{$^2$Cosmology and Gravity Group, Department of Mathematics and Applied Mathematics, University of Cape Town, Rondebosch 7701, South Africa}
\affiliation{$^3$South African Astronomical Observatory, Observatory 7925, Cape Town, South Africa}

\begin{abstract}
In this work, we study various geometric properties of embedded spacelike hypersurfaces in $1+1+2$ decomposed spacetimes with a preferred spatial direction, denoted $e^{\mu}$, which are orthogonal to the fluid flow velocity of the spacetime and admit a proper conformal transformation. To ensure non-vanishing and positivity of the scalar curvature of the induced metric on the hypersurface, we impose that the scalar curvature of the conformal metric is non-negative and that the associated conformal factor $\varphi$ satisfies $\hat{\varphi}^2+2\hat{\hat{\varphi}}>0$, where $\hat{\ast}$ denotes derivative along the preferred spatial direction. Firstly, it is demonstrated that such hypersurface is either of Einstein type or the spatial twist vanishes on it, and that the scalar curvature of the induced metric is constant. It is then proved that if the hypersurface is compact and of Einstein type and admits a proper conformal transformation, then the hypersurface must be isomorphic to the $3$-sphere, where we make use of some well known results on Riemannian manifolds admitting conformal transformations. If the hypersurface is not of Einstein type and have nowhere vanishing sheet expansion, we show that this conclusion fails. However, with the additional conditions that the scalar curvatures of the induced metric and the conformal metric coincide, the associated conformal factor is strictly negative and the third and higher order derivatives of the conformal factor vanish, the conclusion that the hypersurface is isomorphic to the $3$-sphere follows. Furthermore, additional results are obtained under the conditions that the scalar curvature of a metric conformal to the induced metric is also constant. Finally, we consider some of our results in context of locally rotationally symmetric spacetimes and show that, if the hypersurfaces are compact and not of Einstein type, then under specified conditions the hypersurface is isomorphic to the $3$-sphere, where we constructed explicit examples of proper conformal Killing vector fields along $e^{\mu}$.
\end{abstract}

\pacs{}

\maketitle

\section{Introduction}

Conformal symmetries on Lorentzian manifolds is a well studied subject and even more so for Riemannian manifolds \cite{tb1,g1,g2,hs1,si1,t1,mo1,mo2,yam,yan1,yan2,yan3,yan4}. In the Lorentzian case, conformal Killing vector (tensor) fields have well defined kinematic and local geometric interpretations (see \cite{rm1,rm3,rm4,ac4,her1,pet1,mt1} and references therein). Higher order conformal Killing tensors have fewer studies exposing explicit physical interpretations. General results have been obtained, but explicit applications to spacetimes is less common. The works by Crampin \cite{mc1,mc2} and De Groote \cite{gro1} have extensively studied second and higher order special conformal Killing tensors. In particular, for example, the work by Crampin, \cite{mc2}, showed that the Hamilton-Jacobi equations for a Riemannian manifold admitting a conformal Killing tensor with vanishing torsion, can be solved by separation of variables. De Groote focused on special conformal Killing tensors and in \cite{gro1} it was shown that, for the class of spacetimes of Petrov D type, the only special conformal Killing tensor admitted by the spacetime was the constant Killing tensor. 

Conformal geometry in the context of global analysis on manifolds is, on the other hand, well understood. This, however, has been mostly confined to purely mathematical results. There have indeed been several works on conformal geometry via global analysis for embedded hypersurfaces. A. R. Gover and his co-authors \cite{ar1,ar2,ar3} have developed an approach to studying conformally compactified geometries, wherein new conformal invariants were constructed to study obstructions to conformal compactness. Another work worth mentioning was carried out by M. A. Akivis and V. V. Goldberg, \cite{vv1}, where the authors studied the geometry of hypersurfaces in pseudoconformal spaces under the Darboux mapping. Of interest here are works by Yamabe \cite{yam}, Goldberg \cite{si1,si2}, Yano \cite{yan2,si2,yan5} and Obata \cite{yan5,mo2,mo3,mo4}, where the authors studied conformal changes to a given metric on a Riemannian manifold and established conditions under which these manifolds are isometric to the sphere. 

Embedded hypersurfaces in spacetimes play a crucial role in various aspects of general relativity. For example, the Hamiltonian formulation of general relativity, \cite{hamme}, examines foliation of a spacetime by constant time slices for a choice of a global time function in the spacetime. Also, the existence of constant mean curvature spacelike hypersurfaces in spacetimes have consequences for the structure of singularities \cite{sos}, as well as its use in the proof of the positive mass theorem \cite{pmt}. Spacelike hypersurfaces of constant mean curvature are also ubiquitous in the study of cosmological models. Indeed, studying the geometry of hypersurfaces with symmetries in spacetimes is certainly of interest to a wide range of relativists. 

The focus of this work is on conformal changes to the metric of embedded \(3\)-dimensional embedded hypersurfaces in \(1+1+2\) decomposed spacetimes and the implications for the geometry of these hypersurfaces, where these hypersurfaces assume a particular form of the Ricci tensor. (As will be seen, the particular form of the Ricci tensor we will prescribe represents constant time slices, where the time function parametrizes the observers' world line.) This lies at the interface of differential geometry, geometric analysis and general relativity, and hence of interest to experts in all three fields. The approach to be utilised should bring out the relationship between conformal geometry and physical quantities (kinematic and geometric) specifying the hypersurfaces, and the role these quantities play in the characterisation of these hypersurfaces. While the works by Goldberg \cite{si1,si2}, Yano \cite{yan2,si2,yan5} and Obata \cite{yan5,mo2,mo3,mo4} considered a similar problem, here we approach the problem in a covariant way. We point out that the \(1+1+2\) covariant approach was recently used to study conformal symmetries in the class of locally rotationally symmetric metrics \cite{say2}, with the successful obtention of some interesting results, which tied the existence of conformal symmetries to an extremal value of the heat flux, a very interesting relationship between thermodynamics and symmetries of the spacetime. This is perhaps the first use of this semitetrad approach with regards to conformal symmetries. This work however ties the global geometry of embedded hypersurfaces to conformaly symmetries. 

This paper is outlined as follows: Section \ref{sec2} introduces the \(1+1+2\) covariant formalism that is employed in this paper, providing sufficient details as necessary to enable those not familiar with the formalism (or relativity in general) to follow the rest of the paper. In Section \ref{sec3}, we present the form of the Ricci tensor for the class of hypersurfaces to be considered. The concomitant tensors relevant to this work are written down in their covariant form. In Section \ref{sec4} the behaviour of the associated quantities under conformal transformations are considered and the conformal quantities computed. Section \ref{sec5} provides a discussion on the characterization of the hypersurfaces. In Section \ref{sec6}, some properties of the hypersurfaces, given that these hypersurfaces admit a conformal transformation, are examined and discussion on the form of conformal Killing vectors along the preferred spatial direction \(e^{\mu}\) are obtained. Section \ref{sec7} presents some results on the geometry of these hypersurfaces under conformal transformation. In Section \ref{sec8}, we discuss and analyze some of our results in context of the well known locally rotationally symmetric class of spacetimes, where there exists some preferred spatial direction. We conclude with discussion of the results in Section \ref{sec9}, and present potential future avenue of research.

\section{\(1+1+2\) spacetime decomposition}\label{sec2}

Let \(M\) be a \(4\)-dimensional spacetime, and let \(U^{\mu}\) be a \(4\)-vector in \(M\). Given a preferred unit timelike vector field \(u^{\mu}\) in \(M\) (usually chosen as the unit tangent to the observer's congruence), one splits \(U^{\mu}\) as
\begin{eqnarray*}
U^{\mu}&=&Uu^{\mu} + U^{\langle \mu \rangle }.
\end{eqnarray*}
The scalar \(U\) is the component parallel to the vector \(u^{\mu}\), and \(U^{\langle \mu \rangle }\) is the projected \(3\)-vector (the angle bracket indicates that the vector is symmetric and trace-free) projected via the tensor \(h_{\mu}^{\ \nu}\equiv g_{\mu}^{\ \nu}+u_{\mu}u^{\nu}\) which results from splitting of the metric on \(M\) \cite{ggff2}. This splitting indeed decomposes the covariant derivative of the vector \(u^{\mu}\) as
\begin{eqnarray}\label{mmmn}
\nabla_{\mu}u_{\nu}=-u_{\mu}\mathcal{A}_{\nu}+\frac{1}{3}h_{\mu\nu}\Theta+\sigma_{\mu\nu},
\end{eqnarray}
and the energy momentum tensor decomposes as
\begin{eqnarray}
T_{\mu\nu}=\rho u_{\mu}u_{\nu} + 2q_{(\mu}u_{\nu)} +ph_{\mu\nu} + \pi_{\mu\nu}.
\end{eqnarray}

The quantity \(\mathcal{A}_{\mu}\) is the acceleration vector; \(\Theta\equiv D_{\mu}u^{\mu}\) is the expansion; \(\sigma_{\mu\nu}=D_{\langle \nu}u_{{\mu}\rangle}\) is the projected symmetric trace-free shear tensor; \(\rho\equiv T_{\mu\nu}u^{\mu}u^{\nu}\) is the energy density; \(q_{\mu}=-h_{\mu}^{\ \nu}T_{\nu\gamma}u^{\gamma}\) is the \(3\)-vector defining the heat flux; \(p\equiv\left(1/3\right)h^{\mu\nu}T_{\mu\nu}\) is the isotropic pressure; and the tensor \(\pi_{\mu\nu}\) defines the anisotropic stress.

Given a unit spatial direction \(e^{\mu}\) orthogonal to \(u^{\mu}\), one may split the \(3\)-space into a direction along \(e^{\mu}\) and the remaing \(2\)-space, with both \(u^{\mu}\) and \(e^{\mu}\) orthogonal to this \(2\)-space (the \(2\)-space is sometimes referred to as the ``\(2\)-sheet" in the literature). This further splitting results in the decomposition of the metric on \(M\),
\begin{eqnarray}
N_{\mu\nu}=g_{\mu\nu}+u_{\mu}u_{\nu}-e_{\mu}e_{\nu}.
\end{eqnarray} 
The tensor \(N_{\mu\nu}\) projects \(2\)-vectors orthogonal to \(u^{\mu}\) and \(e^{\mu}\) onto the \(2\)-surface (note that \(N^{\ \mu}_{\mu}\) is \(2\)), with \(u^{\mu}N_{\mu\nu}=e^{\mu}N_{\mu\nu}=0\). The vectors \(u^{\mu}\) and \(e^{\mu}\) are normalized so that \(u^{\mu}u_{\mu}=-1\) and \(e^{\mu}e_{\mu}=1\). The splitting results in the following derivatives:

\begin{itemize}
\item For an arbitrary tensor \(S^{\mu..\nu}_{\ \ \ \ \gamma..\delta}\), one defines the \textit{covariant time derivative} (or simply the dot derivative)  along the observers' congruence of \(S^{\mu..\nu}_{\ \ \ \ \gamma..\delta}\) as \(\dot{S}^{\mu..\nu}_{\ \ \ \ \gamma..\delta}\equiv u^{\sigma}\nabla_{\sigma}S^{\mu..\nu}_{\ \ \ \ \gamma..\delta}\).

\item For an arbitrary tensor \(S^{\mu..\nu}_{\ \ \ \ \gamma..\delta}\) one defines the fully orthogonally \textit{projected covariant derivative} \(D\) with the tensor \(h_{\mu\nu}\) as \(D_\sigma S^{\mu..\nu}_{\ \ \ \ \gamma..\delta}\equiv h^{\mu}_{\ \rho}h^{\eta}_{\ \gamma}...h^{\nu}_{\ \tau}h^{\iota}_{\ \delta}h^{\lambda}_{\ \sigma}\nabla_{\lambda}S^{\rho..\tau}_{\ \ \ \ \eta..\iota}\).

\item Given a \(3\)-tensor \(\psi^{\mu..\nu}_{\ \ \ \ \gamma..\delta}\) the spatial derivative along the vector field \(e^{\mu}\) (simply called the \textit{hat derivative}) is given by \(\hat{\psi}_{\mu..\nu}^{\ \ \ \ \gamma..\delta}\equiv e^{\sigma}D_{\sigma}\psi_{\mu..\nu}^{\ \ \ \ \gamma..\delta}\).

\item Given a \(3\)-tensor \(\psi^{\mu..\nu}_{\ \ \ \ \gamma..\delta}\) the projected spatial derivative on the \(2\)-sheet (projection by the tensor \(N_{\mu}^{\ \nu}\)), called the \textit{delta derivative}, is given by \(\delta_\sigma\psi_{\mu..\nu}^{\ \ \ \ \gamma..\delta}\equiv N_{\mu}^{\ \rho}..N_{\nu}^{\ \tau}N_{\eta}^{\ \gamma}..N_\iota^{\ \delta}N_{\sigma}^{\ \lambda}D_{\lambda}\psi_{\rho..\tau}^{\ \ \ \ \eta..\iota}\).
\end{itemize}

For any \(3\)-vector \(\psi^{\mu}\), one may split \(\psi^{\mu}\) into a scalar part \(\Psi\), parallel to \(e^{\mu}\) and a vector part \(\Psi^{\mu}\), lying in the \(2\)-sheet, which is orthogonal to \(e^{\mu}\), i.e.
\begin{eqnarray}
\psi^{\mu}=\Psi e^{\mu}+\Psi^{\mu}.
\end{eqnarray}

The shear, electric and magnetic Weyl tensors can be written respectively as

\begin{subequations}
\begin{align}
\sigma_{\mu\nu}&=\Sigma\left(e_{\mu}e_{\nu}-\frac{1}{2}N_{\mu\nu}\right)+2\Sigma_{(\mu}e_{\nu)}+\Sigma_{\mu\nu},\\
E_{\mu\nu}&=\mathcal{E}\left(e_{\mu}e_{\nu}-\frac{1}{2}N_{\mu\nu}\right)+2\mathcal{E}_{(\mu}e_{\nu)}+\mathcal{E}_{\mu\nu},\\
H_{\mu\nu}&=\mathcal{H}\left(e_{\mu}e_{\nu}-\frac{1}{2}N_{\mu\nu}\right)+2\mathcal{H}_{(\mu}e_{\nu)}+\mathcal{H}_{\mu\nu},
\end{align}
\end{subequations}
and the full covariant derivatives of the vectors \(u^{\mu}\) and \(e^{\mu}\) are given respectively by \cite{cc1}

\begin{subequations}
\begin{align}
\nabla_{\mu}u_{\nu}&=-\left(\mathcal{A}_{\nu}+\mathcal{A}e_{\nu}\right)u_{\mu} + e_{\mu}e_{\nu}\left(\frac{1}{3}\Theta + \Sigma\right)+\Omega\varepsilon_{\mu\nu}+\frac{1}{2} N_{\mu\nu}\left(\frac{2}{3}\Theta -\Sigma\right)+e_{\mu}\left(\Sigma_{\nu}+\varepsilon_{\nu\delta}\Omega^{\delta}\right)+\Sigma_{\mu\nu}\notag\\
&+\left(\Sigma_{\mu}-\varepsilon_{\mu\delta}\Omega^{\delta}\right)e_{\nu},\label{4}\\
\nabla_{\mu}e_{\nu}&=-\mathcal{A}u_{\mu}u_{\nu} -u_{\mu}\alpha_{\nu}+\left(\frac{1}{3}\Theta + \Sigma\right)e_{\mu}u_{\nu} +\frac{1}{2}\phi N_{\mu\nu}+\left(\Sigma_{\mu}-\varepsilon_{\mu\delta}\Omega^{\delta}\right)u_{\nu}+e_{\mu}a_{\nu}+\xi\varepsilon_{\mu\nu}+\zeta_{\mu\nu},\label{444}
\end{align}
\end{subequations}
where \(\varepsilon_{\mu\nu}=\varepsilon_{\mu\nu\delta}e^{\delta}=u^{\gamma}\eta_{\gamma\mu\nu\delta}e^{\delta}\) is the \(2\)-dimensional alternating Levi-Civita tensor, \(a_{\mu}=\hat{e}_{\mu}\) is the acceleration of the normal vector \(e^{\mu}\), \(\phi=\delta_{\mu}e^{\mu}\) is the sheet expansion, \(\xi=\frac{1}{2}\varepsilon^{\mu\nu}\delta_{\mu}e_{\nu}\) is the sheet twist and \(\zeta_{\mu\nu}=\delta_{\mu}e_{\nu}\) is the shear of \(e^{\mu}\). We also have the following relations:

\begin{eqnarray*}
\begin{split}
\dot{e}^{\mu}&=\mathcal{A}e^{\mu}+\alpha^{\mu}\\
\omega^{\mu}&=\Omega e^{\mu}+\Omega^{\mu},\\
q^{\mu}&=Q e^{\mu}+Q^{\mu},
\end{split}
\end{eqnarray*}
where \(\omega^{\mu}\) is the rotation vector. (Full details of this formalism can be found in \cite{cc1} and associated references.)

\section{The Ricci and concomitant tensors}\label{sec3}

In this section, we shall briefly discuss an approach to specifying hypersurfaces through the prescription of the Ricci tensor of the associated induced metric on the hypersurface. We will then write out the concomitant tensors required for the rest of this work, in terms of the \(1+1+2\) covariant quantities.

Let \(\mathcal{T}\) be a codimension \(1\) properly embedded submanifold in a \(1+1+2\) decomposed spacetime. Denote by \(g_{\mu\nu}\) the Lorentzoan metric on \(M\), and \(h_{\mu\nu}\) the induced metric on \(h_{\mu\nu}\). One obtains the curvature quantities on \(\mathcal{T}\) from those of \(M\) through the following steps \cite{haw1}:

\begin{enumerate}

\item Define a general vector \(Y^{\mu}\), which is normal to \(\mathcal{T}\);

\item Decompose the metric \(g_{\mu\nu}\) into the sum of the first  fundamental form \(h_{\mu\nu}\) on \(\mathcal{T}\) and the symmetric \(2\)-tensor \(Y_{\mu}Y_{\nu}\):

\begin{eqnarray}\label{3.1}
g_{\mu\nu}=h_{\mu\nu}\pm Y_{\mu}Y_{\nu};
\end{eqnarray}

\item Obtain the second fundamental form on \(\mathcal{T}\) as 

\begin{eqnarray}\label{3.2}
\chi_{\mu\nu}=h^{\delta}_{\ (\mu}h^{\gamma}_{\ \nu)}\nabla_{\delta}Y_{\gamma};
\end{eqnarray}

\item The Ricci tensor on \(\mathcal{T}\) with respect to \(h_{\mu\nu}\) is obtained as 

\begin{eqnarray}\label{3.3}
R'_{\mu\nu}=\left(R_{\delta\gamma}\mp R^{\alpha}_{\ \delta\beta\gamma}Y_{\alpha}Y^{\beta}\right)h^{\delta}_{\ \mu}h^{\gamma}_{\ \nu}\pm\chi\chi_{\mu\nu}\mp\chi^{\sigma}_{\ \mu}\chi_{\sigma\nu},
\end{eqnarray}

\end{enumerate}
where we have used the ``prime" to denote the Ricci tensor on \(\mathcal{T}\) and the unprimed Ricci and curvature tensors are those of the ambient spacetime \(M\) (this is the convention we shall follow throughout this work). We will also use \(``D"\) to denote the compatible connection on \(\mathcal{T}\). The scalar \(\mathcal{X}\) is just the trace of \eqref{3.2}. Whether \(\mathcal{T}\) is timelike or spacelike specifies the choice of sign (``+" or ``-" respectively). 

The curvature tensor on \(\mathcal{T}\) can be computed using

\begin{eqnarray}\label{3.4}
R'^{\mu}_{\ \nu\delta\gamma}=R^{\alpha}_{\ \delta\beta\gamma}h^{\delta}_{\ \mu}h^{\gamma}_{\ \nu}h^{\delta}_{\ \mu}h^{\gamma}_{\ \nu}\pm\chi^{\mu}_{\ \delta}\chi_{\nu\sigma}\mp\chi^{\mu}_{\ \sigma}\chi_{\nu\delta},
\end{eqnarray}
and the scalar curvature is just

\begin{eqnarray}\label{3.6}
R'&=\left(R\pm\chi^2\right)\mp\left(R_{\mu\nu}Y^{\mu}Y^{\nu}+\chi_{\mu\nu}\chi^{\mu\nu}\right).
\end{eqnarray}

In this work, we will be interested in a class of codimension \(1\) hypersurfaces with the following particular form of the Ricci tensor:

\begin{eqnarray}\label{3.7}
R'_{\mu\nu}=\alpha e_{\mu}e_{\nu}+\beta N_{\mu\nu},
\end{eqnarray}
with the scalar curvature given by

\begin{eqnarray}\label{3.8}
R'=\alpha+2\beta.
\end{eqnarray}
The form of the Ricci tensor \eqref{3.7} represents spacelike hypersurfaces at an instant of time. Physically, they represent a co-moving observer's rest space in the spacetime. The choice of the Ricci tensor is largely motivated by those of constant time spacelike slices in the class locally rotationally symmetric solutions, to which we will apply some of the geometric results we shall obtain in the subsequent sections. (This class of spacetimes contains well studied solutions like the Oppenheimer dust model, Lemaitre-Tolman-Bondi solutions, Schwarzschild solution, etc., and aspects of hypersurfaces in these spacetimes have also been studied. For example, these hypersurfaces, if of Cauchy type, can be used to specify initial data by which to evolve the quantities specifying the spacetime (see for example \cite{say1}).

The curvature tensor and the cotton tensor can be expressed respectively as

\begin{subequations}
\begin{align}
R'_{\mu\nu\delta\gamma}&=\left(N_{\delta\nu}+e_{\delta}e_{\nu}\right)\biggl[\frac{\alpha}{2}N_{\gamma\mu}+\left(\beta-\frac{\alpha}{2}\right)e_{\gamma}e_{\mu}\biggr]-\left(N_{\gamma\nu}+e_{\gamma}e_{\nu}\right)\biggl[\frac{\alpha}{2}N_{\delta\mu}+\left(\beta-\frac{\alpha}{2}\right)e_{\delta}e_{\mu}\biggr]\notag\\
&+\left(N_{\delta\mu}+e_{\delta}e_{\mu}\right)\left(\beta N_{\gamma\nu}+\alpha e_{\gamma}e_{\nu}\right)-\left(N_{\gamma\mu}+e_{\gamma}e_{\mu}\right)\left(\beta N_{\delta\nu}+\alpha e_{\delta}e_{\nu}\right),\label{3.11}\\
C'_{\mu\nu\delta}&=D_{\delta}R'_{\mu\nu}-D_{\nu}R'_{\mu\delta}+\frac{1}{4}\left(h_{\mu\delta}D_{\nu}R'-h_{\mu\nu}D_{\delta}R'\right)\nonumber\\
&=2\left(\alpha-\beta\right)\left(e_{\mu}D_{[\delta}e_{\nu]}+e_{[\nu}D_{\delta]}e_{\mu}\right) +2\left[e_{\mu}e_{[\nu}D_{\delta]}+\frac{1}{4}h_{\mu[\delta}D_{\nu]}\right]\alpha+2\left[\left(h_{\mu[\nu}-e_{\mu}e_{[\nu}\right)D_{\delta]}+\frac{1}{2}h_{\mu[\delta}D_{\nu]}\right]\beta,\label{3.12}
\end{align}
\end{subequations}

The scalars \(\alpha\) and \(\beta\) can be expressed in terms of well defined scalar quantities on the spacetimes. Explicitly, we compute them as

\begin{subequations}
\begin{align}
\alpha&=\frac{2}{3}\left(\rho+\Lambda\right)+\mathcal{E}+\frac{1}{2}\Pi-\left(\frac{1}{3}\Theta+\Sigma\right)\left(\frac{2}{3}\Theta-\Sigma\right),\label{3.9}\\
\beta&=\frac{2}{3}\rho-\frac{1}{2}\left(\mathcal{E}+\frac{1}{2}\Pi\right)-\frac{1}{2}\left(\frac{2}{3}\Theta-\Sigma\right)\left(\frac{2}{3}\Theta+\frac{1}{2}\Sigma\right)+2\Omega^2,\label{10}
\end{align}
\end{subequations}
with \(\Lambda\) being the cosmological constant. It is important to mention that in general, there are additional terms in the Ricci tensor including \(u_{\mu}u_{\nu}\), mixed terms in \(u_{\mu}\) and \(e_{\mu}\), as well as terms constructed from products of \(2\)-vectors in the spacetime and the unit vectors \(u_{\mu}\) and \(e_{\mu}\). So the case we are considering is restrictive, but nonetheless does capture hypersurfaces in some well known spacetimes as was just mentioned.

For the particular class of hypersurfaces under consideration here, the first fundamental form on \(\mathcal{T}\) is just the projector

\begin{eqnarray}\label{3.10}
h_{\mu\nu}&=N_{\mu\nu}+e_{\mu}e_{\nu}.
\end{eqnarray}

\section{Behavior of the tensors under conformal rescalings}\label{sec4}

In this section, we interpret the behavior of various curvature quantities under conformal transformation in terms of the spacetime covariant variables.

For a given manifold, if two metrics exist on the manifold that are related to each other by of a scale factor, then the metrics are said to be conformal to each other. A necessary and sufficient condition for such a relation is the existence of an infinitesimal transformation generated by some vector field. Such transformations have implications fo the global geometry of the manifolds and will be used throughout the results of this paper. 

Let \(\mathfrak{X}\left(H\right)\) denote the set of smooth vector fields on a hypersurface \(\mathcal{T}\), and let \(v\in\mathfrak{X}\left(H\right)\). Then, \(\mathcal{T}\) is said to admit a conformal transformation if

\begin{eqnarray}\label{4.1}
\left(\mathcal{L}_{v}-2\varphi\right)h_{\mu\nu}=0,
\end{eqnarray}
where \(\varphi\in C^{\infty}\left(T\right)\) is a smooth function on \(\mathcal{T}\), and the operator \(\mathcal{L}_{v}\) denotes the Lie derivative along the vector field \(v\). The metric generated by the transformation is given as

\begin{eqnarray}\label{4.2}
\tilde{h}_{\mu\nu}=e^{2\varphi}h_{\mu\nu}.
\end{eqnarray}
If \(\varphi=0\), then the transformation is referred to as an \textit{isometry}. If \(\varphi\) is a non-zero constant, then the transformation is called a \textit{homothety}. And if \(\varphi\) is non-constant, then the transformation is said to be a \textit{proper} conformal transformation. 

Now, define the following quantities \cite{yan5}

\begin{eqnarray*}
\varphi_{\mu}&=D_{\mu}\varphi,\quad\nu=e^{-\varphi},\quad\nu_{\mu}=D_{\mu}\nu,\quad\Delta\varphi=h^{\mu\nu}D_{\mu}D_{\nu}\varphi,\\
\varphi_{\mu\nu}&=D_{\mu}\varphi_{\nu}-\varphi_{\mu}\varphi_{\nu}+\frac{1}{2}\varphi_{\sigma}\varphi^{\sigma}h_{\mu\nu},\quad\varphi_{\mu}^{\mu}=\Delta\varphi+\frac{1}{2}\varphi_{\mu}\varphi^{\mu},
\end{eqnarray*}
where \(\Delta\) is the \(3\)-dimensional Laplacian. Their \(1+1+2\) covariant expressions are given by

\begin{subequations}
\begin{align}
\varphi_{\mu}&=\hat{\varphi}e_{\mu}+\delta_{\mu}\varphi,\label{4.3}\\
\nu_{\mu}&=-e^{-\varphi}\left(\hat{\varphi}e_{\mu}+\delta_{\mu}\varphi\right)=-\nu\varphi_{\mu},\label{4.4}\\
\Delta\varphi&=\hat{\hat{\varphi}}+\phi\hat{\varphi}-\left(a^{\mu}-\delta^{\mu}\right)\delta_{\mu}\varphi,\label{4.5}\\
\varphi_{\mu\nu}&=\left(\hat{\hat{\varphi}}-\frac{1}{2}\hat{\varphi}^2+\delta_{\sigma}\varphi\delta^{\sigma}\varphi\right)e_{\mu}e_{\nu}-2\hat{\varphi}e_{(\mu}\delta_{\nu)}\varphi+\left(D_{\mu}-\delta_{\mu}\varphi\right)\delta_{\nu}\varphi+\left(\delta_{\mu}\hat{\varphi}+\hat{\varphi}D_{\mu}\right)e_{\nu}+\frac{1}{2}\left(\hat{\varphi}^2+\delta_{\sigma}\varphi\delta^{\sigma}\varphi\right)N_{\mu\nu},\label{4.6}\\
\varphi_{\mu}^{\ \mu}&=\hat{\hat{\varphi}}+\left(\phi+\frac{1}{2}\hat{\varphi}\right)\hat{\varphi}+\left(\frac{1}{2}\delta_{\mu}\varphi+\delta_{\mu}-a_{\mu}\right)\delta^{\mu}\varphi.\label{4.7}
\end{align}
\end{subequations}
We also define the following tensor

\begin{eqnarray}\label{4.8}
\mathcal{G}_{\mu\nu}&=R'_{\mu\nu}-\frac{1}{3}R'h_{\mu\nu},
\end{eqnarray}
whose covariant expression is given by

\begin{eqnarray}\label{4.9}
\mathcal{G}_{\mu\nu}=\frac{1}{3}\left(\alpha-\beta\right)\left(2e_{\mu}e_{\nu}-N_{\mu\nu}\right).
\end{eqnarray}
The tensor \(\mathcal{G}_{\mu\nu}\) satisfies \cite{yan5}

\begin{subequations}
\begin{align}
h^{\mu\nu}\mathcal{G}_{\mu\nu}&=0,\label{4.10}\\
D^{\mu}\mathcal{G}_{\mu\nu}&=\frac{1}{6}D_{\nu}R'.\label{4.11}
\end{align}
\end{subequations}
Clearly \eqref{4.10} is satisfied as the second parenthesis of \eqref{4.9} is zero. Furthermore, under conformal transformation of the induced metric, the Ricci tensor, the Ricci scalar and \(\mathcal{G}_{\mu\nu}\) transform as \cite{yan5}

\begin{subequations}
\begin{align}
\tilde{R}'_{\mu\nu}&=R'_{\mu\nu}-\varphi_{\mu\nu}-\varphi_{\sigma}^{\ \sigma}h_{\mu\nu}\label{4.15},\\
\tilde{R}'&=e^{-2\varphi}\left(R'-4\varphi_{\mu}^{\ \mu}\right)\label{4.16},\\
\tilde{\mathcal{G}}_{\mu\nu}&=\mathcal{G}_{\mu\nu}-\left(D_{\mu}\varphi_{\nu}-\varphi_{\mu}\varphi_{\nu}\right)+\frac{1}{3}\left(\Delta\varphi-\varphi_{\sigma}\varphi^{\sigma}\right)h_{\mu\nu}\label{4.17},
\end{align}
\end{subequations}
which can be expressed in the covariant way as

\begin{subequations}
\begin{align}
\tilde{R}'_{\mu\nu}&=-\left(\hat{\hat{\varphi}}+\hat{\varphi}^2+\frac{1}{2}\phi\hat{\varphi}+\left(\delta_{\sigma}-a_{\sigma}\right)\delta^{\sigma}\varphi+\delta_{\sigma}\varphi\delta^{\sigma}\varphi-\alpha\right)N_{\mu\nu}-\left(2\hat{\hat{\varphi}}+\left(\delta_{\sigma}-a_{\sigma}\right)\delta^{\sigma}\varphi+\delta_{\sigma}\varphi\delta^{\sigma}\varphi-\beta\right)e_{\mu}e_{\nu}\notag\\
&+2\hat{\varphi}e_{(\mu}\delta_{\nu)}\varphi-e_{\mu}\widehat{\delta_{\nu}\varphi}-e_{\nu}\delta_{\mu}\hat{\varphi}+\left(\delta_{\mu}\varphi-\delta_{\mu}\right)\delta_{\nu}\varphi,\label{4.18}\\
\tilde{R}'&=e^{-2\varphi}\left[R'-2\left(\left(\hat{\varphi}^2+\delta_{\mu}\varphi\delta^{\mu}\varphi\right)+2\left(\hat{\hat{\varphi}}+\delta_{\mu}\delta^{\mu}\varphi-a_{\mu}\delta^{\mu}\varphi\right)\right)\right],\label{4.19}\\
\tilde{\mathcal{G}}_{\mu\nu}&=-\frac{1}{3}\left[\left(\alpha-\beta\right)-\hat{\hat{\varphi}}-\left(\phi-\hat{\varphi}\right)\hat{\varphi}+\left(a^{\sigma}-\delta^{\sigma}+\delta^{\sigma}\varphi\right)\delta_{\sigma}\varphi\right]N_{\mu\nu}+\left(\hat{\varphi}D_{\mu}+\delta_{\mu}\hat{\varphi}\right)e_{\nu}+2\hat{\varphi}e_{(\mu}\delta_{\nu)}\varphi\notag\\
&+\frac{2}{3}\left[\left(\alpha-\beta\right)-\hat{\hat{\varphi}}+\left(\frac{1}{2}\phi+\hat{\varphi}\right)\hat{\varphi}-\frac{1}{2}\left(a^{\sigma}-\delta^{\sigma}+\delta^{\sigma}\varphi\right)\delta_{\sigma}\varphi\right]e_{\mu}e_{\nu}+\left(D_{\mu}+\delta_{\mu}\varphi\right)\delta_{\nu}\varphi,\label{4.20}
\end{align}
\end{subequations}
where the overhead '\textit{tilde}` notation denotes quantities associated to the metric \(\tilde{h}_{\mu\nu}\). We also define two important scalars

\begin{subequations}
\begin{align}
\mathcal{G}_{\mu\nu}\varphi^{\mu}\varphi^{\nu}&=\frac{1}{3}\left(\alpha-\beta\right)\left(2\hat{\varphi}^2-\delta_{\mu}\varphi\delta^{\mu}\varphi\right),\label{4.21}\\
\mathcal{G}_{\mu\nu}\nu^{\mu}\nu^{\nu}&=e^{-2\varphi}\mathcal{G}_{\mu\nu}\varphi^{\mu}\varphi^{\nu}\notag\\
&=\frac{1}{3}e^{-2\varphi}\left(\alpha-\beta\right)\left(2\hat{\varphi}^2-\delta_{\mu}\varphi\delta^{\mu}\varphi\right)\\
\Longrightarrow \nu^{-2}\mathcal{G}_{\mu\nu}\nu^{\mu}\nu^{\nu}&=\mathcal{G}_{\mu\nu}\varphi^{\mu}\varphi^{\nu}.\label{4.23}
\end{align}
\end{subequations}
These quantities will be very crucial when we study the geometry of the hypersurfaces under consideration in Section \ref{sec7}.

Now, the last relation \eqref{4.11}, when explicitly written gives the following:

\begin{eqnarray}\label{4.12}
\left[\frac{2}{3}\widehat{\left(\alpha-\beta\right)}+\left(\alpha-\beta\right)\phi\right]e_{\mu}+\left(\alpha-\beta\right)a_{\mu}-\frac{1}{3}\delta_{\mu}\left(\alpha-\beta\right)=0,
\end{eqnarray}
which, by contraction with \(e^{\mu}\), gives

\begin{eqnarray}\label{4.13}
\widehat{\left(\alpha-\beta\right)}+\frac{3}{2}\phi\left(\alpha-\beta\right)=0.
\end{eqnarray}

Indeed, if \(\mathcal{T}\) is of Einstein type, then \eqref{4.13} is clearly satisfied. Otherwise, \(\alpha-\beta\) satisfies

\begin{eqnarray}\label{4.14}
\left(\alpha-\beta\right)=\exp\left(-\frac{3}{2}\int \phi\ d\chi\right),
\end{eqnarray}
where integration is carried out along the integral curves of \(e^{\mu}\) (\(\chi\) parametrizes the integral curves along \(e^{\mu}\)), and \(\phi\neq0\). Hence, \(\alpha-\beta>0\) (remember we are assuming the covariant variables are finite). If \(\phi=0\), then simply

\begin{eqnarray}\label{hh}
\widehat{\left(\alpha-\beta\right)}=0.
\end{eqnarray}
In this case \(\alpha-\beta\) is a constant with no restriction on the sign. As will be seen in Section \ref{sec7}, desired results will require that this constant be non- negative.

Notice that the \(2\)-vector \(\delta^{\mu}\varphi\) is spacelike, and hence \(\delta_{\mu}\varphi\delta^{\mu}\varphi\geq0\). Hence, the sum

\begin{eqnarray}\label{4.24}
\hat{\varphi}^2+\delta_{\mu}\varphi\delta^{\mu}\varphi\geq0,
\end{eqnarray}
with the sum being zero if and only if the transformation is not a proper conformal transformation (an isometry or a homothety). We therefore make the following observation: in the case of vanishing sheet terms, i.e. \(a_{\mu}=\delta_{\mu}\varphi=0\), whenever we have a conformal transformation that is proper, for

\begin{subequations}
\begin{align}
\hat{\hat{\varphi}}\geq0\quad\mbox{or}\quad\frac{1}{2}\hat{\varphi}^2\geq \ \vline\ \hat{\hat{\varphi}}\ \vline\ ,
\end{align}
\end{subequations}  
imposing that the scalar curvature associated to the conformally transformed metric to be non-negative, i.e. \(\tilde{R}'\geq0\), ensures that the scalar curvature for the metric induced from the ambient spacetime is positive, i.e. \(R'>0\). The necessity of \(R'\) being strictly positive is to avoid \(R'\) vanishing somewhere on the hypersurface, as the hypersurface will necessarily be flat there (the curvature tensor will vanish).

In the subsequent sections, we present the calculations and results of this work.

\section{Characterisation of the hypersurfaces}\label{sec5}

We begin this section by providing a certain useful characterisation of the hypersurfaces under consideration. We will state this as the following proposition:

\begin{proposition}\label{pro1}
Let \(M\) be a \(1+1+2\) decomposed spacetime, and \(\mathcal{T}\) an embedded codimension 1 hypersurface with Ricci tensor of the form \eqref{3.7}. Then, either one of the following is true:

\begin{enumerate}

\item \(\mathcal{T}\) is of Einstein type; or

\item \(\mathcal{T}\) is non-twisting.

\end{enumerate}
\end{proposition}

\begin{proof}
A straightforward contraction of \eqref{3.12} by \(e^{\mu}\varepsilon^{\nu\delta}\) gives

\begin{eqnarray}\label{5.2}
\left(\alpha-\beta\right)\xi=0.
\end{eqnarray}
Hence, either \(\alpha=\beta\), in which case \(\mathcal{T}\) is of Einstein type or, the hypersurface is non-twisting, i.e. \(\xi=0\).\qed
\end{proof}

We emphasize that the two cases of Proposition \eqref{pro1} are not mutually exclusive, i.e. one may have an Einstein type hypersurface with vanishing spatial twist. Interestingly, even if \(\mathcal{T}\) is not of Einstein type, whenever \(\mathcal{T}\) has zero sheet expansion, i.e. \(\phi=0\), the condition that \(\hat{\alpha}=0\) would imply that the induced metric on \(\mathcal{T}\) is of constant scalar curvature. To see this, notice that from the contracted Bianchi identities, we have

\begin{eqnarray}\label{5.3}
\hat{\alpha}e_{\mu}+\left(\alpha-\beta\right)\left(a_{\mu}+\phi e_{\mu}\right)=\frac{1}{2}\left(\hat{R}'e_{\mu}+\delta_{\mu}\alpha\right),
\end{eqnarray}
which we contract by \(e^{\mu}\) to obtain

\begin{eqnarray}\label{5.4}
\hat{\alpha}+\left(\alpha-\beta\right)\phi=\frac{1}{2}\hat{R}'.
\end{eqnarray}
Conversely, if the scalar curvature is constant and \(\hat{\alpha}\) vanishes, then for \(\phi\neq0\) \(\mathcal{T}\) must be of Einstein type, i.e. \(\alpha-\beta=0\).

\section{Some properties of the hypersurfaces under conformal transformation}\label{sec6}

Before we begin this section, we specify the following two conditions that will be assumed throughout the rest of the paper:

\begin{enumerate}

\item  For a conformally transformed metric \(\tilde{h}_{\mu\nu}=e^{2\varphi}h_{\mu\nu}\) on \(\mathcal{T}\), the scalar curvature \(\tilde{R}'\) associated to \(\tilde{h}_{\mu\nu}\) is non-negative. (In any case, this condition will be explicitly stated whenever a proposition or theorem is presented in this work.)

\item To keep some of our calculations simplified we will also impose that, for any smooth function \(\psi\in C^{\infty}\left(\mathcal{T}\right)\), we have that \(\psi\) is constant on the \(2\)-sheet, i.e. \(\delta_{\mu}\psi=0\).
\end{enumerate}

We recall the well known result by Yamabe, relating a given metric to a conformally equivalent one on a compact Riemannian manifold:

\begin{theorem}[Yamabe]\label{thea1}
Let \(N\) be a smoothly differentiable and compact Riemannian \(n\geq3\) dimensional manifold. Then, for any given metric on \(N\), there always exists a Riemannian metric with constant scalar curvature, which is conformal to the given metric.
\end{theorem}

As is well known, the originaly proof by Yamabe was flawed, and the modification by Trudinger \cite{tru1} and the subsequent resolution by Scheon \cite{scr} require some restriction on the Yamabe invariant. In particular, it is required that the Yamabe invariant of the manifold be bounded above by that of a sphere of the same dimension. Indeed, it is known that if the scalar curvature of \(h_{\mu\nu}\) is non-negative and not identically zero, \(h_{\mu\nu}\) can be deformed to a metric of constant positive scalar curvature. Hence, if one imposes that the scalar curvature of the conformally transformed metric is non-negative and that the associated conformal factor \(\varphi\) satisfies the inequality

\begin{eqnarray}\label{uoy1}
\hat{\varphi}^2+2\hat{\hat{\varphi}}>0,
\end{eqnarray}
(or the more restrictive condition \(\hat{\hat{\varphi}}\geq0\)), then \(R'\) is strictly positive. These two conditions will therefore imply that the Yamabe equation is solvable. 

Firstly, let us proceed to prove the following result:

\begin{proposition}\label{pro2}
If \(M\) is a \(4\)-dimensional \(1+1+2\) decomposed spacetime, and \(\mathcal{T}\) an embedded \(3\)-dimensional manifold in \(M\) with Ricci tensor of the form \eqref{3.7}, then, the scalar curvature of \(\mathcal{T}\) is constant, i.e.

\begin{eqnarray}\label{proeq1}
D_{\mu}R'=\hat{R}'=0.
\end{eqnarray}

\end{proposition} 

\begin{proof}
A straightforward substitution of \eqref{4.13} into \eqref{5.4} simplifies to

\begin{eqnarray}\label{proeq2}
\frac{1}{3}\left(\hat{\alpha}+2\hat{\beta}\right)=\frac{1}{2}\hat{R}'.
\end{eqnarray}
We note the parenthesized term on the left hand side of \eqref{proeq2} as just the derivative of \eqref{3.8}, and hence the result follows.\qed
\end{proof}
Notice that, in obtaining the result of Proposition \ref{pro2}, we do not assume compactness. Hence, we could potentially have non-compact examples to the Yamabe problem in the class of hypersurfaces considered here, under the assumption \(\tilde{R}'\geq0\) and that \eqref{uoy1} holds. 

Now, suppose \(\tilde{h}_{\mu\nu}=e^{2\varphi}h_{\mu\nu}\) is a metric conformal to the induced metric \(h_{\mu\nu}\) on \(\mathcal{T}\), and let \(\tilde{R}'\) denote the scalar curvature associated to \(\tilde{h}_{\mu\nu}\). Considering the problem of finding the conditions under which \(\tilde{R}'\) is constant can be rephrased as a problem that examines \textit{under which conditions the constancy of the scalar curvature of the induced metric is an invariant property under conformal transformation.} 

Indeed, the constancy of \(\tilde{R}'\) imposes the following condition on the scalar curvature associated to the induced metric: the derivative of \eqref{4.19}) gives

\begin{eqnarray}\label{5.6}
\hat{R}'&=-2\left[\hat{\varphi}\left(R'-2\left(2\hat{\hat{\varphi}}-\hat{\varphi}^2\right)\right)+2\left(\hat{\hat{\hat{\varphi}}}+\hat{\varphi}\hat{\hat{\varphi}}\right)\right],
\end{eqnarray}
and hence, since \(\tilde{R}'\) is constant, we have that either \(\hat{\varphi}=0\) (in this case the transformation is not a proper conformal transformation) or

\begin{eqnarray}\label{5.7}
R'=-2\hat{\varphi}^{-1}\left(\hat{\hat{\hat{\varphi}}}-\hat{\varphi}\hat{\hat{\varphi}}-\hat{\varphi}^3\right).
\end{eqnarray}
Therefore, whenever the scalar curvature of the induce metric satisfies \eqref{5.7}, then a metric conformal to the induce metric, with associated conformal factor \(\varphi\), has constant scalar curvature.

Now, we recall that under the assumption that \(\tilde{R}'\geq0\) and that \eqref{uoy1} holds, if the transformation is proper we must have \(R'>0\). This then provides the following required restriction on the conformal factor:

\begin{eqnarray}\label{5.9}
\hat{\varphi}^{-1}\left(\hat{\hat{\hat{\varphi}}}-\hat{\varphi}\hat{\hat{\varphi}}-\hat{\varphi}^3\right)<0.
\end{eqnarray}
For small \(\varphi\), so that its derivatives of order three (3) or higher is negligible (we will write this condition as \(D^{(n\geq3)}_{\mu}\varphi=0\)), the above condition simply reduces to

\begin{eqnarray}\label{5.9s}
\hat{\hat{\varphi}}+\hat{\varphi}^2>0.
\end{eqnarray}
(In any case, if the transformation is not a proper conformation transformation, then \(R'=\tilde{R}'=0\).) It therefore follows that

\begin{proposition}\label{pro3}
Let \(M\) be a \(4\)-dimensional \(1+1+2\) decomposed spacetime, and \(\mathcal{T}\) an embedded \(3\)-dimensional manifold in \(M\) with Ricci tensor of the form \eqref{3.7}, and suppose \(\mathcal{T}\) admits a conformal transformation. Let \(\tilde{h}_{\mu\nu}=e^{2\varphi}h_{\mu\nu}\) be a metric conformal to \(h_{\mu\nu}\) such that the scalar curvature \(\tilde{R}'\) associated to \(\tilde{h}_{\mu\nu}\) is non-negative, and \eqref{uoy1} holds. If the transformation is a proper conformal transformation and \(\hat{\tilde{R}}'=0\), then the conformal factor \(\varphi\) satisfies \eqref{5.9}. And whenever \(\varphi\) is such that \(D^{(n\geq3)}_{\mu}\varphi=0\), then \eqref{5.9s} holds.
\end{proposition} 

Note that the condition of \eqref{5.9s} implies the condition \eqref{uoy1}, as long as \(\hat{\hat{\varphi}}\geq0\). Therefore, under the assumptions of Proposition \ref{pro3} and that \(D^{(n\geq3)}_{\mu}\varphi=0\) and \(\hat{\hat{\varphi}}\geq0\), if the scalar curvature \(\tilde{R}'\) is non-negative, then that of the induced metric must be strictly positive. But notice that, under the assumption \(D^{(n\geq3)}_{\mu}\varphi=0\), by substituting \eqref{5.7} into \eqref{4.19} we can simplify to obtain

\begin{eqnarray}\label{5.10}
\tilde{R}'=-2e^{-2\varphi}\hat{\hat{\varphi}}.
\end{eqnarray}
By assumption \(\tilde{R}'\geq0\), which imposes that \(\hat{\hat{\varphi}}\leq0\). Hence, if we start by imposing that \(\hat{\hat{\varphi}}\geq0\), then we will have that \(\hat{\hat{\varphi}}\) must be zero, in which case \(\tilde{R}'\). Therefore the assumption \(\hat{\hat{\varphi}}\geq0\) will be relaxed. 

As a consequence of Proposition \ref{pro3} we have the following corollary:

\begin{corollary}\label{cor1}
Let \(M\) is a \(4\)-dimensional \(1+1+2\) decomposed spacetime, and \(\mathcal{T}\) an embedded \(3\)-dimensional manifold in \(M\) with Ricci tensor of the form \eqref{3.7}, and suppose \(\mathcal{T}\) admits a proper conformal transformation. Let \(\tilde{h}_{\mu\nu}=e^{2\varphi}h_{\mu\nu}\) be a metric conformal to \(h_{\mu\nu}\) such that the scalar curvature \(\tilde{R}'\) associated to \(\tilde{h}_{\mu\nu}\) is constant and non-negative, and \eqref{uoy1} holds. If \(D^{(n\geq3)}_{\mu}\varphi=0\) and \(\varphi\) is at least a \(C^2\) function, then the conformal factor \(\varphi\) satisfies

\begin{eqnarray}\label{5.11}
\hat{\hat{\varphi}}\leq0,
\end{eqnarray}
along with the constraint

\begin{eqnarray}\label{5.11s}
\hat{\varphi}^2>\ \vline\ \hat{\hat{\varphi}}\ \vline\ .
\end{eqnarray}

\end{corollary} 
(The constraint \eqref{5.11s} is to ensure that \eqref{5.9s} holds.)

We have seen that, for all metrics conformal to the induced metric \(h_{\mu\nu}\), with constant scalar curvature, the scalar curvatures of the induced metric and the conformal metrics are entirely specified by the conformal factor, provided that the conformal factor satisfies \eqref{5.9} and \eqref{5.11}.

As an example, suppose we consider the following vector field parallel to the preferred spatial direction

\begin{eqnarray}\label{5.12}
X^{\mu}=\gamma e^{\mu}.
\end{eqnarray}
It can easily be checked that if \(\gamma\) satisfies

\begin{subequations}
\begin{align}
A\gamma&=\varphi,\label{peewee1}\\
\dot{\gamma}-\left(\frac{1}{3}\Theta+\Sigma\right)\gamma&=0,\label{peewee2}\\
\hat{\gamma}&=\varphi,\label{5.13}\\
\phi\gamma&=2\varphi,\label{5.14}
\end{align}
\end{subequations}
for \(\varphi\) a smooth and non-constant function, then \(X^{\mu}\) is a proper conformal Killing vector field. Indeed, such \(\gamma\) solves

\begin{eqnarray}\label{5.15}
\hat{\gamma}-\frac{1}{2}\phi\gamma=0,
\end{eqnarray}
with the constraints

\begin{subequations}
\begin{align}
\phi-2A&=0,\label{peewee3}\\
\frac{1}{3}\Theta+\Sigma&=0,\label{peewee4}
\end{align}
\end{subequations}
(where we have noted that on the spacelike slices under consideration here, \(\dot{\gamma}=0\)). Solution to \eqref{5.15} is guaranteed, and takes the form

\begin{eqnarray}\label{5.16}
\gamma=\exp\left(\frac{1}{2}\int \phi\ d\chi\right)>0,
\end{eqnarray}
for \(\phi\neq0\) (if \(\phi=0\), then we simply have the case of a Killing vector that is a constant multiple of \(e^{\mu}\)). The associated conformal factor is given as

\begin{eqnarray}\label{5.17}
\begin{split}
\varphi&=\frac{1}{2}\phi\gamma\nonumber\\
&=\frac{1}{2}\phi\exp\left(\frac{1}{2}\int \phi\ d\chi\right).
\end{split}
\end{eqnarray}
It then follows that, under the assumptions of Proposition \ref{pro3}, if \(\tilde{R}'\) is constant, then, \(R'\) and \(\tilde{R}'\) can explicitly be written in terms of the sheet expansion respectively as

\begin{subequations}
\begin{align}
R'&=\left[\left(\hat{\phi}^2+\phi^2\hat{\phi}+\frac{1}{4}\phi^4\right)\exp\left(\frac{1}{2}\int\phi\ d\chi\right)+2\left(\hat{\hat{\phi}}+\frac{3}{2}\phi\hat{\phi}+\frac{1}{4}\phi^3\right)\right]\exp\left(\frac{1}{2}\int\phi\ d\chi\right),\label{5.18}\\
\tilde{R}'&=-\left(\hat{\hat{\phi}}+\frac{3}{2}\phi\hat{\phi}+\frac{1}{4}\phi^3\right)\exp\left(\frac{1}{2}\int\phi\ d\chi + \phi\exp\left(\frac{1}{2}\int\phi\ d\chi\right)\right),\label{5.19}
\end{align}
\end{subequations}
provided \(\phi\) satisfies

\begin{eqnarray}\label{5.20}
\hat{\hat{\phi}}+\frac{3}{2}\phi\hat{\phi}+\frac{1}{4}\phi^3\leq0.
\end{eqnarray}
The last equation, \eqref{5.20}, is obtained from \eqref{5.11}. It indeed follows that, if \(\phi\) has an upper bound, so does \(R'\) and \(\tilde{R}'\). This is very interesting as it implies that, compactness of these hypersurfaces could simply be specified by the sheet expansion.

\section{Geometry of the hypersurfaces}\label{sec7}

Let us now consider the global geometry of the hypersurfaces under consideration, given that they admit a conformal transformation. We shall use some well known results from Riemannian geometry. We start by stating a result due to Goldberg, Yano and Obata \cite{si1,yan2,mo2,mo3}.

\begin{theorem}\label{thea3}
Let \(N\) be a compact \(n\geq2\)-dimensional Riemannian manifold with constant scalar curvature, and suppose \(N\) admits a proper conformal transformation such that \(\left(\mathcal{L}_{\nu}-2\varphi\right)h_{\mu\nu}=0\). Then a necessary and sufficient condition for \(N\) to be isometric to a sphere is 

\begin{eqnarray}\label{5.21}
\int_{\mathcal{T}}\mathcal{G}_{\mu\nu}\varphi^{\mu}\varphi^{\nu}dV=0,
\end{eqnarray}
where \(dV\) denotes the volume element of \(\mathcal{T}\).
\end{theorem} 

Now, the integrand in \eqref{5.21} is given in \eqref{4.21} as (we rewrite it here for immediate reference).

\begin{eqnarray*}
\frac{2}{3}\left(\alpha-\beta\right)\hat{\varphi}^2.
\end{eqnarray*}
Clearly, if \(\mathcal{T}\) is of Einstein type, then \eqref{5.21} always holds. However, if \(\mathcal{T}\) is not Einstein, then we have already shown that \(\left(\alpha-\beta\right)>0\) (here we must insist that the sheet expansion on \(\mathcal{T}\) is non-negative for otherwise \(\left(\alpha-\beta\right)\) can be a negative constant), and since the transformation is proper, \(\hat{\varphi}^2>0\). Hence, \((2/3)\left(\alpha-\beta\right)\hat{\varphi}^2>0\). This allows us to conclude the following:

\begin{theorem}\label{thea4}
Let \(M\) be a \(4\)-dimensional \(1+1+2\) decomposed spacetime, and \(\mathcal{T}\) a compact embedded \(3\)-dimensional manifold in \(M\) with Ricci tensor of the form \eqref{3.7}. Furthermore, suppose \(\mathcal{T}\) admits a proper conformal transformation. Let \(\tilde{h}_{\mu\nu}=e^{2\varphi}h_{\mu\nu}\) be a metric conformal to \(h_{\mu\nu}\) such that the scalar curvature \(\tilde{R}'\) associated to \(\tilde{h}_{\mu\nu}\) is non-negative. If \(\mathcal{T}\) is Einstein, then \(\mathcal{T}\) is isometric to the \(3\)-sphere. 

\end{theorem} 

Indeed, if \(\mathcal{T}\) is non-Einstein, it is clear that the integral \eqref{5.21} can be negative or positive, in which case the conclusion of Theorem \ref{thea4} cannot follow. However, under certain conditions the global geometry of a non-Einstein type \(\mathcal{T}\) can be specified. For the non-Einstein case, with \(\phi=0\) and \(\hat{\alpha}\neq0\), one simply require that the following condition be satisfied in order for the conclusion of Theorem \ref{thea4} to hold on \(\mathcal{T}\):

\begin{eqnarray}\label{hft}
\alpha-\beta=c,
\end{eqnarray}
for some positive constant \(c\). 

The case with non-vanishing sheet is a little more involving. Let us recall the following result due to Goldberg \cite{si2}, Obata \cite{mo2,mo3} and Yano \cite{yan2}.

\begin{theorem}\label{thea5}
Let \(N\) be a compact \(n\geq2\)-dimensional Riemannian manifold with constant scalar curvature, and suppose \(N\) admits a proper conformal transformation such that \(\left(\mathcal{L}_{\nu}-2\varphi\right)h_{\mu\nu}=0\) such that \(\tilde{R}'=R'\), where \(\tilde{R}'\) is the scalar curvature associated to the conformal metric. If

\begin{eqnarray}\label{5.22}
\int_{\mathcal{T}} \nu^{1-n}\mathcal{G}_{\mu\nu}\nu^{\mu}\nu^{\nu}dV\geq0,
\end{eqnarray}
where \(dV\) denotes the volume element of \(\mathcal{T}\), then \(N\) is isometric to a sphere.
\end{theorem} 

Firstly, for the case considered in this work, the integrand of \eqref{5.22} reduces to that of \eqref{5.21} as was seen from \eqref{4.21} and \eqref{4.23}. Clearly, as long as the sheet expansion stays positive all over \(\mathcal{T}\), the condiition \eqref{5.22} always holds, since the integrand is either zero or positive, and zero if and only if \(\mathcal{T}\) is Einstein. Now, we have assumed non-negativity of \(\tilde{R}'\). If \eqref{uoy1} holds, then, from \eqref{4.19} the condition that \(\tilde{R}'=R'\) requires the following restriction in terms of the conformal factor \(\varphi\):

\begin{eqnarray}\label{5.23}
\left(1-e^{2\varphi}\right)>0,
\end{eqnarray}
(recall that the condition \eqref{uoy1} ensures a strictly positive \(R'\)), where we are assuming that the transformation is proper. Hence, we must have \(\varphi<0\). 

Also, we point out that we must insist that the sheet expansion is non-zero if \eqref{hft} is not satisfied, since otherwise we would have \(\alpha-\beta\) as an arbitrary constant, as was earlier discussed. In the case that this constant is negative, then the result fails since the integrand of \eqref{5.22} is strictly negative. This then leads us to the following result.

\begin{theorem}\label{thea6}
Let \(M\) be a \(4\)-dimensional \(1+1+2\) decomposed spacetime, and \(\mathcal{T}\) a compact embedded \(3\)-dimensional manifold in \(M\) with Ricci tensor of the form \eqref{3.7}, with nowhere vanishing sheet expansion. Furthermore, suppose \(\mathcal{T}\) admits a proper conformal transformation. If \(\tilde{h}_{\mu\nu}=e^{2\varphi}h_{\mu\nu}\) is a metric conformal to \(h_{\mu\nu}\) such that the scalar curvature \(\tilde{R}'\) associated to \(\tilde{h}_{\mu\nu}\) is non-negative and \(\tilde{R}'=R'\), with \(\varphi<0\) and \eqref{uoy1} satisfied, then \(\mathcal{T}\) is isometric to the \(3\)-sphere. 

\end{theorem} 

Again it is very important to emphasize that, without the condition \(\varphi<0\) the conclusion of the above theorem is not possible.

Considering Theorem \ref{thea6} in context of the case discussed earlier, where the transformation was induced by the vector field \(X^{\mu}=\gamma e^{\mu}\) parallel to \(e^{\mu}\), we see that the condition on the conformal factor, \(\varphi<0\), is equivalent to the statement that the hypersurface \(\mathcal{T}\) has negative sheet expansion. So, for example, in the case that \(\phi\) is continuously decreasing along \(e^{\mu}\), \eqref{5.20} is valid as long as

\begin{eqnarray}\label{5.24}
\Big|\hat{\hat{\phi}}+\frac{1}{4}\phi^3\Big|\geq\frac{3}{2}\phi\hat{\phi}.
\end{eqnarray}

Now, let us consider the case discussed earlier where the scalar curvature associated to the conformally transformed metric is also constant (where the condition \(D^{(n\geq3)}_{\mu}\varphi=0\) is imposed). In this case, \(\tilde{R}'=R'\) can be specified via the following second order non-linear equation in the conformal factor \(\varphi\):

\begin{eqnarray}\label{5.25}
\hat{\hat{\varphi}}+\left(1+e^{-2\varphi}\right)^{-1}\hat{\varphi}^2=0.
\end{eqnarray}
As we are interested in proper conformal transformations, we will assume that \(\hat{\hat{\varphi}}\neq0\), since from \eqref{5.25} this would force \(\hat{\varphi}=0\), which in turn would imply that the transformation is homothetic. Hence, by Corollary \ref{cor1} we must have

\begin{eqnarray}\label{5.26}
\hat{\hat{\varphi}}<0.
\end{eqnarray}
But this would require that 

\begin{eqnarray}\label{5.26s1}
\left(1+e^{-2\varphi}\right)>0,
\end{eqnarray} 
which will always hold. However, it will appear that we have encountered a problem here: by \eqref{5.23} we have that \(e^{-2\varphi}>1\). From \eqref{5.26s1}, this gives the estimate

\begin{eqnarray}\label{5.26s2} 
\left(1+e^{-2\varphi}\right)>2.
\end{eqnarray}
Thus, we have from \eqref{5.25} that

\begin{eqnarray}\label{5.26s3}
\begin{split} 
\hat{\varphi}^2&=-\left(1+e^{-2\varphi}\right)\hat{\hat{\varphi}}\\
&<-2\hat{\hat{\varphi}},
\end{split}
\end{eqnarray}
and hence

\begin{eqnarray}\label{5.26s4} 
\hat{\varphi}^2+2\hat{\hat{\varphi}}<0.
\end{eqnarray}
The above equation means that the strict positivity of the scalar curvature of the induced metric is not guaranteed by \eqref{uoy1}. But notice that, if \(\hat{\tilde{R}}'=0\) and \(D^{(n\geq3)}_{\mu}\varphi\), then to ensure strict positivity of \(R'\) requires the following condition be satisfied:

\begin{eqnarray}\label{5.26s5} 
\hat{\hat{\varphi}}+\frac{3}{4}\hat{\varphi}^2 > 0,
\end{eqnarray}
and since \(\hat{\hat{\varphi}}<0\) we further require that

\begin{eqnarray}\label{5.26s6} 
\frac{3}{4}\hat{\varphi}^2>\ \vline\ \hat{\hat{\varphi}}\ \vline\ .
\end{eqnarray}

In addition, notice that \eqref{5.26s2} implies that \(\varphi<0\). It therefore follows that the conditions \eqref{5.25}, \eqref{5.26}, \eqref{5.26s5}, and \eqref{5.26s6} ensure that \(\varphi<0\) and \(\tilde{R}'=R'\). This then allows us to state the following result:

\begin{theorem}\label{thea7}
Let \(M\) be a \(4\)-dimensional \(1+1+2\) decomposed spacetime, and \(\mathcal{T}\) a compact embedded \(3\)-dimensional manifold in \(M\) with Ricci tensor of the form \eqref{3.7}, with nowhere vanishing sheet expansion. Furthermore, suppose \(\mathcal{T}\) admits a proper conformal transformation and \(\tilde{h}_{\mu\nu}=e^{2\varphi}h_{\mu\nu}\) is a metric conformal to \(h_{\mu\nu}\) such that the scalar curvature \(\tilde{R}'\) associated to \(\tilde{h}_{\mu\nu}\) is constant, \eqref{uoy1} is satisfied and \(D^{(n\geq3)}_{\mu}\varphi=0\). If the conditions
\begin{enumerate}

\item \(\hat{\hat{\varphi}}<0\); 

\item \(\hat{\hat{\varphi}}+\left(1+e^{-2\varphi}\right)^{-1}\hat{\varphi}^2=0\); and

\item \(\hat{\hat{\varphi}}+\frac{3}{4}\hat{\varphi}^2 > 0\)

\end{enumerate}
hold, then \(\mathcal{T}\) is isometric to the \(3\)-sphere. 

\end{theorem}

\section{Application to locally rotationally symmetric spacetimes}\label{sec8}

In this section, we apply some of our results of the previous section to locally rotationally symmetric (LRS) spacetimes. Indeed, LRS spacetimes are \(1+1+2\) decomposed with all vector and tensor quantities vanishing, and hence their constant time spacelike slices are precisely of the form \eqref{3.7}, and hence the results herein apply given that the assumptions of the various propositions and theorems hold on these slices. (The slices will similarly be denoted by \(\mathcal{T}\) as has been done throughout this work.) We are therefore also interested in some constraints on the various scalar quantities on spacelike slices, given the results we have obtained. For simplicity, we will consider those LRS solutions with vanishing shear. First we define what these solutions are.

\begin{definition}
A spacetime \(M\) is said to be \textbf{locally rotationally symmetric (LRS)} if, at each point \(p\in M\), there exists a continuous isotropy group generating a multiply transitive isometry group on \(M\) \cite{ggff2}, with the metric given by

\begin{eqnarray}\label{jan29191}
\begin{split}
ds^2&=-A^2dt^2 + B^2d\chi^2 + F^2 dy^2 + \left[\left(F\bar{D}\right)^2+ \left(Bh\right)^2 - \left(Ag\right)^2\right]dz^2+ \left(A^2gdt - B^2hd\chi\right)dz,
\end{split}
\end{eqnarray}
where \(A,B,F\) are functions of \(t\) and \(\chi\), \(\bar{D}^2\) is a function of \(y\) and \(k\), with \(k\) specifying the geometry of the \(2\)-surfaces, and \(g,h\) being functions of \(y\) only. 
\end{definition}
These spacetimes are specified by the below set of scalar variables \cite{cc1}
\begin{eqnarray*}
\lbrace{A,\Theta,\phi, \Sigma, \mathcal{E}, \mathcal{H}, \rho, p, \Pi, Q, \Omega, \xi\rbrace}. 
\end{eqnarray*}
The well known LRS II class of spacetimes generalizing spherically symmetric solutions to the Einstein field equations, is the limiting case of the general LRS class of spacetimes defined above, with \(g=h=0\).

For this section, we will assume that the spacetimes have non-vanishing sheet expansion so as to apply earlier obtained results. 

The field equations for these spacetimes can be written as a collection of evolution and propagation equations and their mixture \cite{cc1}. For the particular case considered here, these equations are

\begin{itemize}

\item \textit{Evolution}
\begin{subequations}
\begin{align}
\frac{2}{3}\dot{\Theta}-\dot{\Sigma}&=\mathcal{A}\phi- \frac{1}{2}\left(\frac{2}{3}\Theta-\Sigma\right)^2 - 2\Omega^2 + \mathcal{E}-\frac{1}{2}\Pi - \frac{1}{3}\left(\rho+3p\right),\label{6.1}\\
\dot{\phi}&=\left(\frac{2}{3}\Theta-\Sigma\right)\left(\mathcal{A}-\frac{1}{2}\phi\right) + 2\xi\Omega+Q,\label{6.2}\\
\dot{\xi}&=-\frac{1}{2}\left(\frac{2}{3}\Theta-\Sigma\right)\xi + \left(\mathcal{A}-\frac{1}{2}\phi\right)\Omega,\label{6.3}\\
\dot{\Omega}&=\mathcal{A}\xi-\left(\frac{2}{3}\Theta-\Sigma\right)\Omega,\label{6.4}\\
\dot{\mathcal{H}}&=-3\xi\mathcal{E}-\frac{3}{2}\left(\frac{2}{3}\Theta-\Sigma\right)\mathcal{H}+\Omega Q,\label{6.5}\\
\dot{\mathcal{E}}-\frac{1}{3}\dot{\rho}+\frac{1}{2}\dot{\Pi}&=3\xi\mathcal{H}+\frac{1}{2}\phi Q+\left(\frac{2}{3}\Theta-\Sigma\right)\left[\frac{1}{2}\left(\rho+p\right)-\frac{3}{2}\left(\mathcal{E}+\frac{1}{6}\Pi\right)\right],\label{6.6}
\end{align}
\end{subequations}
\item \textit{Propagation}
\begin{subequations}
\begin{align}
\frac{2}{3}\hat{\Theta}-\hat{\Sigma}&=\frac{3}{2}\phi\Sigma+2\xi\Omega+Q,\label{6.7}\\
\hat{\phi}&=-\frac{1}{2}\phi^2 + \left(\frac{1}{3}\Theta+\Sigma\right)\left(\frac{2}{3}\Theta-\Sigma\right)+2\xi^2-\frac{2}{3}\rho-\left(\mathcal{E}+\frac{1}{2}\Pi\right),\label{6.8}\\
\hat{\xi}&=-\phi\xi + \left(\frac{1}{3}\Theta+\Sigma\right)\Omega,\label{6.9}\\
\hat{\Omega}&=\left(\mathcal{A}-\phi\right)\Omega,\label{6.10}\\
\hat{\mathcal{H}}&=-\left(3\mathcal{E}+\rho+p-\frac{1}{2}\Pi\right)\Omega-3\phi\mathcal{H}-\xi Q,\label{6.11}\\
\hat{\mathcal{E}}-\frac{1}{3}\hat{\rho}+\frac{1}{2}\hat{\Pi}&=-\frac{3}{2}\phi\left(\mathcal{E}+\frac{1}{2}\Pi\right)+3\Omega\mathcal{H}-\frac{1}{2}\left(\frac{2}{3}\Theta-\Sigma\right) Q,\label{6.12}
\end{align}
\end{subequations}
\item \textit{Evolution/Propagation}
\begin{subequations}
\begin{align}
\hat{\mathcal{A}}-\dot{\Theta}&=-\left(\mathcal{A}+\phi\right)\mathcal{A}-\frac{1}{3}\Theta^2+\frac{3}{2}\Sigma^2-2\Omega^2+\frac{1}{2}\left(\rho+3p\right),\label{6.13}\\
\dot{\rho}+\hat{Q}&=-\Theta\left(\rho+p\right)-\left(2\mathcal{A}+\phi\right)Q-\frac{3}{2}\Sigma\Pi,\label{6.14}\\
\dot{Q}+\hat{p}+\hat{\Pi}&=-\left(\rho+p\right)\mathcal{A}-\left(\mathcal{A}+\frac{3}{2}\phi\right)\Pi-\left(\frac{4}{3}\Theta+\Sigma\right) Q,\label{6.15}
\end{align}
\end{subequations}
\item \textit{Constraint}
\begin{eqnarray}\label{6.16}
\mathcal{H}=3\Sigma\xi-\left(2\mathcal{A}-\phi\right)\Omega.
\end{eqnarray}

\end{itemize}

As we are concerned with each spacelike slice \(\mathcal{T}\) in the spacetime, the evolution equations \eqref{6.1} to \eqref{6.6} are the following constraints (the dot derivatives vanish), after some rearrangements:

\begin{subequations}
\begin{align}
0&=\mathcal{A}\phi-\frac{1}{2}\left( \frac{2}{3}\Theta-\Sigma\right)^2 - 2\Omega^2 + \mathcal{E}-\frac{1}{2}\Pi - \frac{1}{3}\left(\rho+3p\right),\label{6.17}\\
0&=\frac{1}{2}\left(\frac{2}{3}\Theta-\Sigma\right)\left(2\mathcal{A}-\phi\right) + 2\xi\Omega+Q,\label{6.18}\\
0&=-\frac{1}{2}\left(\frac{2}{3}\Theta-\Sigma\right)\xi + \frac{1}{2}\left(2\mathcal{A}-\phi\right)\Omega,\label{6.19}\\
0&=\mathcal{A}\xi-\left(\frac{2}{3}\Theta-\Sigma\right)\Omega,\label{6.20}\\
0&=-3\xi\mathcal{E}-\frac{3}{2}\left(\frac{2}{3}\Theta-\Sigma\right)\mathcal{H}+\Omega Q,\label{6.21}\\
0&=3\xi\mathcal{H}+\frac{1}{2}\phi Q+\left(\frac{2}{3}\Theta-\Sigma\right)\left[\frac{1}{2}\left(\rho+p\right)-\frac{3}{2}\left(\mathcal{E}+\frac{1}{6}\Pi\right)\right],\label{6.22}
\end{align}
\end{subequations}
on \(\mathcal{T}\). 

Now, assume that \(\mathcal{T}\) a conformal Killing vector field of the form \eqref{5.12}. Then, using \eqref{peewee3}, \eqref{6.19} becomes

\begin{eqnarray}\label{7.a}
0=\left(\frac{2}{3}\Theta-\Sigma\right)\xi.
\end{eqnarray}
Hence, either \(\xi=0\) or \((2/3)\Theta-\Sigma=0\). If we assume the latter, then from \eqref{6.20} we have that

\begin{eqnarray}\label{7.b}
0=\mathcal{A}\xi.
\end{eqnarray}
We rule out \(\mathcal{A}=0\) since \eqref{peewee3} would imply \(\phi=0\Longrightarrow\varphi=0\), the case that the conformal Killing vector is just a Killing vector (note here that the results we are interested in relies on the assumption that the conformal Killing vector is proper.). Hence, from \eqref{7.b} we must have \(\xi=0\). However, the case \((2/3)\Theta-\Sigma=0\), using \eqref{peewee4}, implies that \(\mathcal{T}\) is time-symmetric (\(\Theta=\Sigma=0\)), which is a further severe restriction on \(\mathcal{T}\). So, from \eqref{7.a} we shall immediately assume that \(\xi=0\) and \((2/3)\Theta-\Sigma\neq0\). Of course then \eqref{6.20} gives \(\Omega=0\). Hence, we are in essence working with class II locally rotationally symmetric solutions (the magnetic Weyl scalar \(\mathcal{H}\) also vanishes, as can be seen from the constraint \eqref{6.16}). We can therefore begin our analysis independent of whether \(\mathcal{T}\) is of Einstein type or not (\(\mathcal{T}\) can be of Einstein type with vanishing twist).

Let us look at the restrictions that \eqref{peewee3} and \eqref{peewee4} impose on \(\mathcal{T}\). Firstly, in this case, \(\mathcal{T}\) can neither radiate nor absorb radiation using \eqref{6.18}. From \eqref{6.22} we therefore have 

\begin{eqnarray}\label{7.c}
\left(\rho+p\right)=3\left(\mathcal{E}+\frac{1}{6}\Pi\right),
\end{eqnarray}
which should be satisfied at all points of \(\mathcal{T}\).

We recall that \(\Theta=0\Longrightarrow \Sigma=0\) from \eqref{peewee4}. Hence, let us start by assuming that \(\Theta\neq0\). From \eqref{6.14} we have

\begin{eqnarray}\label{8.a}
\Pi=2\left(\rho+p\right),
\end{eqnarray}
which, from \eqref{7.c} gives \(\mathcal{E}=0\), i.e. \(\mathcal{T}\) is conformally flat. But using \eqref{peewee3}, comparing \eqref{6.8} and \eqref{6.13}, and then comparing the obtained result to \eqref{6.17} we obtain (we have used \eqref{7.c})

\begin{eqnarray}\label{8.b}
\Theta=0\Longrightarrow\Sigma=0,
\end{eqnarray}
and hence \(\mathcal{T}\) is time-symmetric. This then places the following lower bound on the energy density

\begin{eqnarray}\label{8.c}
\rho>-\frac{3}{2}p.
\end{eqnarray}
(If the isotropic pressure is negative, then the energy densiity is strictly positive.) What we then have is a conformally flat time-symmetric hypersurface.

We can construct different conformal Killing vector fields of the form \eqref{5.12} (and in some cases, along with additional conditions imposed on \(\mathcal{T}\)). Here we provide a few examples.

\begin{itemize}

\item Using \eqref{peewee4}, \eqref{6.7} becomes

\begin{eqnarray}\label{7.d}
\hat{\Sigma}=-\frac{1}{2}\phi\Sigma,
\end{eqnarray}
with solution

\begin{eqnarray}\label{7.e}
\Sigma=\exp\left(-\frac{1}{2}\int\phi\ d\chi\right).
\end{eqnarray}
Hence we have a conformal Killing vector field of the form

\begin{eqnarray}\label{7.f}
X^{\mu}=\frac{1}{\Sigma}e^{\mu},
\end{eqnarray}
with associated conformal factor as

\begin{eqnarray}\label{7.g}
\varphi=\frac{1}{2}\frac{\phi}{\Sigma}.
\end{eqnarray}
Indeed, the transformation is proper if and only if

\begin{eqnarray}\label{7.h}
\frac{\hat{\phi}}{\phi}\neq\frac{\hat{\Sigma}}{\Sigma},
\end{eqnarray}
which can be stated as requiring that \(\phi\) is not proportional to \(\Sigma\).  Of course in the time symmetric case, the component of the vector field blows up.

\item Suppose the energy density \(\rho\) is constant on \(\mathcal{T}\). Then, from \eqref{6.12} we have

\begin{eqnarray}\label{7.i}
\widehat{\left(\rho+p\right)}=-\frac{3}{2}\phi\left(\rho+p\right),
\end{eqnarray}
whose solution is

\begin{eqnarray}\label{7.j}
\rho+p=\exp\left(-\frac{3}{2}\int\phi\ d\chi\right).
\end{eqnarray}
We therefore have a conformal Killing vector field of the form

\begin{eqnarray}\label{7.k}
X^{\mu}=\frac{1}{\left(\rho+p\right)^{\frac{1}{3}}}e^{\mu},
\end{eqnarray}
with associated conformal factor as

\begin{eqnarray}\label{7.l}
\varphi=\frac{\phi}{2\left(\rho+p\right)^{1/3}}.
\end{eqnarray}
Again, we see that the transformation is proper if and only if

\begin{eqnarray}\label{7.m}
\frac{\hat{\phi}}{\phi}\neq\frac{\widehat{\left(\rho+p\right)}}{3\left(\rho+p\right)},
\end{eqnarray}
which can be stated as requiring that \(\phi\) is not proportional to \(\left(\rho+p\right)^{1/3}\). (Notice that if the weak energy condition is satisfied, then \(X^{\mu}\) points in the direction of \(e^{\mu}\).)

\item Finally, rewrite \eqref{6.8} as (recalling \eqref{peewee4})

\begin{eqnarray}\label{7.n}
\hat{\phi}=-W\phi^2,
\end{eqnarray}
where we have defined

\begin{eqnarray}\label{7.o}
W=-\frac{1}{2}-\left(\frac{5}{3}\rho+p\right)\phi^{-2}.
\end{eqnarray}
If \(W\) is constant, then \(\phi\) is implicitly given as

\begin{eqnarray}\label{7.p}
\phi=\exp\left(-W\int\phi\ d\chi\right).
\end{eqnarray}
This gives a conformal Killing vector of the form

\begin{eqnarray}\label{7.q}
X^{\mu}=\frac{1}{\phi^{1/(2W)}}e^{\mu},
\end{eqnarray}
with associated conformal factor given as

\begin{eqnarray}\label{7.r}
\varphi=\frac{1}{2}\phi^{1-1/(2W)}.
\end{eqnarray}
The transformation is proper if and only if

\begin{eqnarray}\label{7.s}
\hat{\phi}\neq0\qquad\mbox{and}\qquad W\neq 1/2.
\end{eqnarray}
Interestingly, it turns out that the requirement that \(W\) be constant implies that \eqref{7.i} is satisfied.

\end{itemize} 

We now consider the cases of Proposition \ref{pro1}. indeed, if \(\mathcal{T}\) is of Einstein type, then \(\alpha-\beta=0\) yields

\begin{eqnarray}\label{6.34}
\rho+p=0.
\end{eqnarray}
However, from \eqref{6.17} one then has

\begin{eqnarray}\label{6.35}
\phi^2=-\frac{4}{3}\rho,
\end{eqnarray}
and hence the energy density is non-vanishing and must be negative, which is usually considered unphysical. The anisotropic stress vanishes as well and the spacetime is further restricted. On the other hand, if \(\rho+p\neq0\), then \(\mathcal{T}\) is not of Einstein type.

We have ruled out the first example we provided of a conformal Killing vector field for the case under consideration in this section. If \(\mathcal{T}\) is of Einstein type then the second example can be ruled out as well, and if \(\mathcal{T}\) is not of Einstein type, then the second example indeed holds.

If \(\mathcal{T}\) is of Einstein type, then the third example simply requires that the energy density be constant. The requirement that the transformation be proper is that the sheet expansion is not constant and \(\rho\neq-3/2\). But propagating \eqref{6.35} along \(e^{\mu}\) we have 

\begin{eqnarray}\label{allo1}
\phi\hat{\phi}=-\frac{2}{3}\hat{\rho}=0,
\end{eqnarray}
and since \(\phi\neq0\), we have that \(\hat{\phi}=0\), and hence the transformation is not proper. We will now collect our main results of this section in the below Lemma, Proposition and Corollary.

\begin{lemma}\label{leem}
Let \(M\) be a class II LRS spacetime with nowhere vanishing sheet expansion, and \(\mathcal{T}\) be an embedded hypersurface in \(M\) orthogonal to the fluid flow velocity of \(M\). If \(\mathcal{T}\) admits a conformal Killing vector field along the preferred spatial direction, then \(\mathcal{T}\) is time-symmetric and conformally flat.
\end{lemma}

\begin{proposition}\label{pro7}
Let \(M\) be a class II LRS spacetime with nowhere vanishing sheet expansion. For an embedded hypersurface \(\mathcal{T}\) with constant energy density in \(M\), orthogonal to the fluid flow velocity of \(M\), if \(\mathcal{T}\) is not of Einstein type and 

\begin{enumerate}
\item \(\phi\not\propto\left(\rho+p\right)^{1/3}\); or

\item \(\hat{p}=0\quad\mbox{and}\quad\frac{5}{3}\rho+p\notin\lbrace{\phi^2,\frac{1}{2}\phi^2\rbrace}\),
\end{enumerate}
then \(\mathcal{T}\) admits the conformal Killing vector field \eqref{7.k} or \eqref{7.q} which are proper.
\end{proposition}
It therefore follows that

\begin{corollary}\label{cor33}
Let \(M\) be a class II LRS spacetime with strictly negative sheet expansion, and let \(\mathcal{T}\) be a compact embedded hypersurface with constant energy density in \(M\), orthogonal to the fluid flow velocity of \(M\). Suppose \(\tilde{R'}=R'\), where \(\tilde{R'}\) is the scalar curvature of the conformal metric to \(h_{\mu\nu}\), obtained by transformations generated by the vector fields \eqref{7.k} and \eqref{7.q}. If \(\mathcal{T}\) is not of Einstein type and

\begin{enumerate}
\item \(\phi\not\propto\left(\rho+p\right)^{1/3}\); or

\item \(\hat{p}=0,\quad\frac{5}{3}\rho+p\notin\lbrace{\phi,\frac{1}{2}\phi^2\rbrace} \quad\mbox{and}\quad W=-\frac{1}{4n}\) where \(n\neq0\) is an integer,
\end{enumerate}
then, \(\mathcal{T}\) is isomorphic to the \(3\)-sphere.
\end{corollary}

Indeed, for the LRS II class of spacetimes we have that

\begin{eqnarray}\label{alllo1}
\alpha-\beta=\frac{3}{4}\Pi.
\end{eqnarray}
Therefore, for the case of non-vanishing sheet expansion, the anisotropic stress\(\Pi\) is necessarily strictly positive since \(\alpha-\beta>0\). 

For the form of the CKV considered here, we can explicitly rule out \(\phi=0\) since this would give \(\varphi\). However, it is quite possible to find a proper CKV even if \(\phi\) vanishes on the hypersurface. In this case, one will have to impose that \(\Pi>0\) as a criterion for Proposition \ref{pro7} (and consequently Corollary \ref{cor33}) to hold. Thus, if were are to relax that condition that \(\phi\) is nowhere vanishing, it will be required to explicitly impose that \(\Pi\) is positive.

\section{Discussion}\label{sec9}

For spacetimes admitting a \(1+1+2\) decomposition, we have studied some geometric properties of (smoothly) embedded spacelike hypersurfaces which are orthogonal to the fluid flow velocity and admit a proper conformal transformation. These results are covariant in nature by virtue of the approach employed. We prescribed the form of the Ricci tensor on these hypersurface and expressed its components explicitly in terms of the \(1+1+2\) covariant quantities specifying the spacetimes. A characterization of the hypersurfaces was provided, which allowed us to properly determine the geometry of the hypersurfaces. We focused on the case  that the induced metric has a positive scalar curvature, although most of the results here hold pointwise without this restriction. Firstly, we showed that the Ricci tensor of the induced metric on said hypersurface is constant. This is important as a lot of standard results in conformal geometry of Riemannian manifolds require the existence of a given metric with constant scalar curvature. We then considered the case where the scalar curvature being constant is an invariant property under conformal transformation. In this case, it was shown that the scalar curvatures of both the induced metric and that conformal to the induced metric can be written entirely in terms of the conformal factor, albeit with certain specified constraints on the conformal factor. In particular, this requires that the second derivative of the conformal factor is non-positive, and that derivatives of order three and higher, of the conformal factor vanish. For the purpose of carrying out some explicit calculations, the case of a conformal vector field parallel to the preferred spatial direction was considered with the component of the vector field and the associated conformal factor computed. It turns out that for the particular case considered, the sheet expansion determines the scalar curvatures, and the condition on the conformal factor can be written as a second order nonlinear inequality in the sheet expansion.

We further utilized standards result due to works by Goldberg, Yano and Obata to show that, for compact Einstein type hypersurfaces of those considered in this work, which admit a proper conformal transformation, it is always true that the hypersurface is isomorphic to the three sphere. The necessary and sufficient condition for the hypersurface to be isomorphic to the sphere is specified by the vanishing of a certain integral. Written in terms of the covariant quantities of the \(1+1+2\) decomposition, the integrand is just the product of the square of the derivative of the conformal factor along the preferred spatial direction, and the difference between the non-zero components of the Ricci tensor. It was then shown that if the sheet expansion on the hypersurface is non-vanishing, then the difference between the two non-zero components of the Ricci tensor is positive. Hence, since the transformation is proper, the square of the conformal is positive. Therefore the integral is strictly positive and cannot be zero. In this case it is clear that if the sheet expansion is non-vanishing and the hypersurface is compact but not of Einstein type, then the result does not hold. However, with the additional condition that the Ricci tensor of the induced metric and that of the conformal metric coincide, one simply requires that the integral be non-negative, and that \(\hat{\varphi}^2+2\hat{\hat{\varphi}}\) is strictly positive. Indeed, the additional condition restricts the sign of the conformal factor, i.e. \(\varphi\) should be strictly negative. In this case then, even if the hypersurface is not of Einstein type, the conclusion that the hypersurface is isomorphic to the three sphere still holds. It turns out that if the conformal vector field is parallel to the preferred spatial direction, then the condition that \(\varphi<0\) is equivalent to the condition the \(\mathcal{T}\) has negative sheet expansion.

Some of the results, specifically some of those of Section \ref{sec7}, were then demonstrated for embedded hypersurfaces in the class II spacetimes with local rotational symmetry, where the sheet expansion is nowhere vanishing. It was shown that in LRS II spacetimes, these hypersurfaces, if they admit, for example, a conformal Killing vector field along \(e^{\mu}\), then they are necessarily flat and time symmetric. We first gave explicit examples of how one may construct such conformal Killing vector fields and the conditions to be satisfied were they to be proper. Since the hypersurfaces are time symmetric and conformally flat, this places restriction on the hypersurfaces. For example, it is shown that our results cannot be used to draw conclusion on the global geometry of the hypersurfaces if they are of Einstein type. We then showed that, on the other hand if the hypersurfaces are not of Einstein type, then under certain conditions, some of the examples we constructed of proper conformal Killing vector fields are admitted by the hypersurface. It therefore followed that these hypersurfaces, if compact, must be of spherical geometry.

Indeed, of crucial importance is to note that, once the hypersurfaces are viewed as embedded proper subsets of these \(1+1+2\) decomposed spacetimes, geometric characterization will now also necessarily be tied to physical quantities in the spacetime. The geometry of the hypersurfaces under conformal transformations clearly shows the restriction imposed on these quantities, a fact well captured by the approach employed in this work.

Besides adding to the literature on conformal geometry with applications to spacetimes, this work nicely bridges general relativity and geometric analysis in a way that allows differential geometers to directly apply results of purely mathematical nature to works in theoretical physics, in addition to other applications. It also provides another platform to create potential synergy between the two fields. 

A potential future endeavour would be to apply our approach to hypersurfaces with a more general form of the Ricci tensor. This would allow for treatment of embedded hypersurfaces of arbitrary causal character and not just spacelike ones. Results in such cases would apply to hypersurfaces evolving in time. And if these hypersurfaces admit the structure of a marginally trapped tube, then such work could potentially provide, geometrically, a classification of certain classes of black hole horizons. So, for example, the most general hypersurface in a general \(1+1+2\) decomposed spacetime will have Ricci tensor of the form

\begin{eqnarray*}
R'_{\mu\nu}=\varrho u_{\mu}u_{\nu}+\alpha e_{\mu}e_{\nu}+2\varsigma u_{(\mu}e_{\nu)}+\beta N_{\mu\nu}+\varkappa_{\mu\nu},
\end{eqnarray*}
for scalars \(\alpha,\beta,\varrho\) and \(\varsigma\), where \(\varkappa_{\mu\nu}\) is a \(2\)-tensor composed of the linear sum of products of \(2\)-vectors and the unit vectors \(u^{\mu}\) and \(e^{\nu}\). If one considers the LRS II class of spacetimes (in which case \(\varkappa_{\mu\nu}=0\)), one could potentially classify black hole horizons in solutions like the Lemaitre-Tolman-Bondi and the Oppenheimer-Snyder ones.

\section*{Acknowledgements}

We are extremely grateful to the anonymous referees for their corrections and suggestions that have greatly improved the results and readability of the paper. AS acknowledges that he is supported by the IBS Center for Geometry and Physics, Pohang University of Science and Technology, Grant No. IBS-R003-D1, and the First Rand Bank, through the Department of Mathematics and Applied Mathematics, University of Cape Town, South Africa. PKSD acknowledges that this work was supported by the First Rand Bank.

\end{document}